\documentclass[11pt]{article}
\usepackage[T1]{fontenc}
\usepackage{latexsym,amssymb,amsmath,amsfonts,amsthm}
\usepackage{graphics}
\usepackage{graphicx}
\usepackage{mathrsfs}
\usepackage{subfigure}
\usepackage{color}
\newcommand{\R}{{\mat R}}

\newcommand{\N}{{\mat N}}

\newcommand{\no}{\nonumber}
\newcommand{\be}{\begin{eqnarray}}
\newcommand{\ben}{\begin{eqnarray*}}
\newcommand{\en}{\end{eqnarray}}
\newcommand{\enn}{\end{eqnarray*}}

\newcommand{\pa}{\partial}

\newcommand{\ov}{\overline}
\newcommand{\I}{{\rm Im}}
\newcommand{\Rt}{{\rm Re}}

\newcommand{\wi}{\widehat}

\newcommand{\wid}{\widetilde}

\newcommand{\mat}{\mathbb}
\newcommand{\se}{\setminus}
\newcommand{\ify}{\infty}

\newcommand{\la}{\lambda}

\newcommand{\tr}{\triangle}

\newtheorem{theorem}{Theorem}[section]
\newtheorem{lemma}[theorem]{Lemma}

\newtheorem{definition}[theorem]{Definition}
\newtheorem{remark}[theorem]{Remark}

\definecolor{hw}{rgb}{0,0,0}
\begin{document}
\renewcommand{\theequation}{\arabic{section}.\arabic{equation}}
\title{\bf Locating a complex inhomogeneous medium with an approximate factorization method}
\author{Fenglong Qu\thanks{School of Mathematics and Informational
Science, Yantai University, Yantai, Shandong, 264005, China ({\tt
fenglongqu@amss.ac.cn})}
\and Haiwen Zhang\thanks{NCMIS and Academy of Mathematics and Systems Science, Chinese Academy of Sciences,
Beijing 100190, China ({\tt zhanghaiwen@amss.ac.cn})}}
\date{}

\maketitle

\vspace{.2in}
%
%
%

\begin{abstract}

Consider the inverse problem of scattering of time-harmonic
acoustic waves by an inhomogeneous medium with complex refractive index. We show that an approximate factorization method can be applied to reconstruct the support of the complex inhomogeneous medium from the far-field data. Numerical examples are also provided to illustrate the practicability of the inversion algorithm.

\vspace{.2in}
{\bf Keywords:} Approximate factorization method, inverse scattering, far-field pattern, inhomogeneous medium.
\end{abstract}

\section{Introduction}
\setcounter{equation}{0}

In this paper, we study the inverse problem of recovering an inhomogeneous medium with complex refractive index from the far-field data. This problem occurs in lots of areas of application such as radar and sonar, medical imaging and non-destructing testing. Precisely, let an open bounded obstacle $D$ denote the inhomogeneous medium with a $C^2$-smooth boundary $\pa D$. Assume that $\ov{D}=\bigcup_{j=1}^{{\color{hw}K}} \ov{D}_j$ with $D_{j_1}\cap D_{j_2}=\emptyset$ if $j_1\neq j_2$. Assume further that $D$ is filled with an inhomogeneous material characterized by the refractive index $n(x)\in L^\infty(D)$ with $\Rt[n(x)]>1$ or $\Rt[n(x)]<1$ in $D_l$ $(1\leq l \leq {\color{hw}K})$, $\left|\Rt[n(x)]-1\right|\geq c$ in $D$ for some positive constant $c$, $\I[n(x)]\geq 0$ {\color{hw}in D}, and the exterior $\R^3\se \ov{D}$ is filled with a homogeneous material with the refractive index $n(x)=1$. It should be remarked that we shall  in the current paper consider the case of complex refractive index, that is, there at least exists two subdomains $D_{l_1}, D_{l_2}$ such that $\Rt[n(x)]>1$ in  $D_{l_1}$ and $\Rt[n(x)]<1$ in $D_{l_2}$ with $1\leq l_1\neq l_2 \leq {\color{hw}K}$. Then the scattering of time-harmonic acoustic waves by the complex inhomogeneous medium $D$ can be modeled by the inhomogeneous Helmholtz equation
\be\label{1.1}
\tr u(x) +k^2n(x)u(x)=0 \qquad \text{in}\;\R^3.
\en
Here, {\color{hw} $k>0$ is the wave number and} $u=u^i+u^s$ denotes the total field with
{\color{hw} the incident wave $u^i$ and the scattered field $u^s$, where
$u^s$ satisfies the Sommerfeld radiation condition
\be\label{1.2}
\frac{\pa u^s}{\pa |x|}- iku^s=\mathcal{O}\left(\frac{1}{|x|^2}\right),
\qquad {\rm as}\;\; |x|\to\ify.
\en}

Moreover,
{\color{hw}it is known that the scattered field $u^s$} has the asymptotic behavior \cite{CK2013}
\be\label{1.3}
u^s(x)=\frac{e^{ik|x|}}{4\pi |x|}u_\ify(\wi{x})+\mathcal{O}\Big(\frac{1}{|x|^2}\Big),
\qquad {\rm as}\;\; |x|\to\ify,
\en
uniformly for all $\wi{x}=x/|x|$, where $u_\ify$ is known as the far-field pattern of $u^s$, which is an analytic function
defined on $\mathbb{S}^2:=\{x\in\R^3:|x|=1\}$.
{\color{hw}In the present paper, we consider $u^i$ to be the incident plane wave which is
given by $u^i=u^i(x;d):=e^{ikx\cdot d}$, where $d\in\mathbb{S}^2$ is the incident direction. Accordingly, the total field, the scattered field and the far-field pattern
are denoted as
$u(x;d)$, $u^s(x;d)$ and $u_\infty(\wi{x};d)$, respectively.}

By using a variational approach, it can be easily shown that the problem (\ref{1.1})-(\ref{1.2}) has a unique solution (see, e.g., \cite{CK2013} or \cite{QYZ2017} for the case when $D$ contains buried objects inside). In the current paper, we are interested in the inverse problem of reconstructing the shape and location of the inhomogeneous medium $D$ from a knowledge of the far-field pattern $u_\infty$ for incident plane waves. The uniqueness of this inverse problem has been established in \cite{YZZ} for the case when $n$ is an unknown constant, in \cite{QuIPI} for the case when $n$ is an unknown piecewise constant, and in \cite{IV, KP1998, QYZIP} for other related inverse medium scattering problems.

In this paper, we study the factorization method as an analytic as well as a numerical tool to reconstruct the shape and location of the inhomogeneous medium $D$ with complex refractive index $n(x)$. For the case when  $\Rt[n(x)]>1$ or $\Rt[n(x)]<1$ in $D$, based on a Lippmann-Schwinger integral equation method, \cite{K1999} proved the validity of the factorization method for recovering the inhomogeneous obstacle $D$. Recently, a factorization method has been developed in \cite{YZZ2013} in determining a penetrable obstacle $D$ with unknown buried objects inside in the case when the solution is discontinuous across the interface $\pa D$, that is, $u|_+=u|_-$, $\pa_\nu u|_+=\la\pa_\nu u|_-$ on $\pa D$ for $\la\not=1$. However, the method used in \cite{YZZ2013} can not be applied to the case when the solution is continuous across the  interface $\pa D$, {\color{hw}that is, $\lambda =1$} (see \cite[Remark 2.5]{YZZ2013}). {\color{hw}To overcome this difficulty,} in \cite{QYZ2017} an approximate factorization method was proposed to solve the same inverse problem as that in \cite{YZZ2013} for the case when the solution is continuous across the  interface $\pa D$.
{\color{hw} However, the factorization method in \cite{K1999, QYZ2017, YZZ2013} depends closely on the assumption that $\Rt[n(x)]>1$ or $\Rt[n(x)]<1$ in $D$.
Therefore, the techniques developed in \cite{K1999, QYZ2017, YZZ2013} can not be directly extended to deal with the case when $\ov{D}=\bigcup_{j=1}^{{\color{hw}K}} \ov{D}_j$ with $\Rt[n(x)]>1$ in $D_{l_1}$ and $\Rt[n(x)]<1$ in  $D_{l_2}$ for some $1\leq l_1\neq l_2 \leq {\color{hw}K}$ which is the
case of the inverse problem under consideration.} The reader is referred to \cite{A2005, Hu2013, L2013} for applications of factorization method for the scattering by diffraction gratings, to \cite{L2008} for the photonics and rough surfaces problems, to \cite{K2016} and \cite{YZZ2014} for the cases of the conductive boundary condition and the generalized impedance boundary condition, and to \cite{KR2012, YHXZ, Hu2016} for the fluid-solid interaction problems. See also \cite{Hu2014} for the rigorous mathematical justification of the factorization method with near-field data. For more detailed overview of the factorization method, we refer to the monograph \cite{K2008} and the references therein, where many related inverse problems for different kinds of partial differential equations are studied by using this method.

For the inverse medium scattering problems, there are also lots of different reconstruction algorithms; see, e.g., \cite{K2002} for the music-algorithm method, \cite{P2005} for the singular sources method, \cite{K2012, H2011, ZZ2013} for the iteration method and \cite{CCS2012, C2010, Kim2012, M2005} for the linear sampling method.

In the present paper, we are motivated by \cite{K1999, QYZ2017} to solve the inverse problem of locating the inhomogeneous medium by developing an approximate factorization method in the case when the medium is filled with an inhomogeneous material characterizing by the complex refractive index. Due to the close dependence of the classical factorization method on the complex refractive index in $D$, we attempt to construct a sequence of perturbed operators $F_m$ of the far-field operator $F$ in a suitable way such that $F_m$ satisfies the Range Identity in \cite[Theorem 2.15]{K2008} for each $m\in\N_+$. Consequently, we can reconstruct the shape and location of medium $D$ from the spectral data of $F_m$
for each $m\in\N_+$. Relying on the construction of $F_m$, we can easily show that $\|F_m-F\|_{L^2(\mathbb{S}^2)}\to 0$ as $m\to \infty$. Thus the exact far-field data $F$ can be regraded as a sufficiently small perturbation of $F_{m_0}$ for some large enough $m_0\in \N_+$. This implies that, for
the noise level $\delta$, the noisy operator $F^\delta$ for $F$ is also a small perturbation of the noisy operator $F^\delta_{m_0}$ for $F_{m_0}$. Therefore, the shape and location of the medium $D$ can be numerically reconstructed by using the spectral data of $F$ and $F^\delta$. Numerical examples that carried out later indeed demonstrate the
practicability of the inversion algorithm.

The remaining part of this paper is organized as follows. In section 2, we propose an approximate factorization method for our inverse problem of locating the inhomogeneous medium with complex refractive index. Numerical examples are provided to illustrate the efficiency of the inversion algorithm in section 3.
Some remarks are also given at the end of section 3.

\section{Approximate Factorization Method}\label{se2}
\setcounter{equation}{0}

In this section, we shall develop an approximate factorization method to study the inverse problem in determining the shape and location of an inhomogeneous medium with complex refractive index. For simplicity we only consider the case when $\ov{D}=\ov{D}_1\cup \ov{D}_2$ with $n=n_1$ in $D_1$ satisfying that $\Rt[n_1(x)]-1\geq c$, $\I [n_1(x)]\geq 0$ and $n=n_2$ in $D_2$ satisfying that $\Rt[n_2(x)]-1\leq-c$, $\I [n_2(x)]\geq 0$ for some positive constant $c$.
We first consider the following general problem
\be\label{3.1}
\left\{
  \begin{array}{ll}
    \tr w + k^2w = 0, & \text{in}\;\R^3\se \ov{D},\\
    \tr w + k^2n_1w = k^2(1-n_1)f_1, & \text{in}\; D_1,\\
    \tr w + k^2n_2w = k^2(1-n_2)f_2, & \text{in}\; D_2,\\
    {\color{hw}w|_+-w|_-=0,\;
    \frac{\pa w}{\pa\nu}|_+-\frac{\pa w}{\pa\nu}|_-=0,}\;& {\color{hw}\text{on}\;\pa D_1\cup\pa D_2,}\\
    {\color{hw}\frac{\pa w}{\pa |x|}- ikw=\mathcal{O}\left(\frac{1}{|x|^2}\right),} & {\color{hw}\text{as}\;|x|\to\ify,}
  \end{array}
\right.
\en
where $f_1\in L^2(D_1), f_2\in L^2(D_2)$. Let $u$ be the total field of the problem (\ref{1.1})-{\color{hw}(\ref{1.2})} corresponding to the incident field $u^i=e^{ikx\cdot d}$, it then follows that $w:=u-u^i$ satisfies the problem (\ref{3.1}) with
{\color{hw}$f_1=u^i|_{D_1}$ and $f_2=u^i|_{D_2}$}.
We now introduce the solution operator $G:Y\mapsto L^2(\mathbb{S}^2)$ by
$$G(f_1,f_2)^T=w_\infty,$$
where $Y:=L^2(D_1)\times L^2(D_2)$ and $w_\infty$ is the far-field pattern of the solution $w$ of the problem (\ref{3.1}) with the given data $(f_1,f_2)^T\in Y$.

{\color{hw}For the solution operator $G$, we have the following lemma.}

\begin{lemma}\label{le3.1}
The solution operator $G$ is compact with dense range in $L^2(\mathbb{S}^2)$.
\end{lemma}
\begin{proof}
{\color{hw}Firstly,} the compactness of the operator $G$ follows easily from the interior regularity results of elliptic equations.
{\color{hw}Secondly,} we need to prove the denseness of the range of $G$ in $L^2(\mathbb{S}^2)$, it suffices to show that the $L^2$-adjoint operator $G^*$ of $G$ is injective.

Assume that $w$ is the solution of the problem (\ref{3.1}) with the data $(f_1,f_2)^T\in Y$, {\color{hw}$w_\infty$ is the far-field pattern of the solution $w$} and
$\widetilde{w}$ is the {\color{hw}total field} of the problem (\ref{1.1})-(\ref{1.2}) with the incident field
\ben
\widetilde{w}^i(y)=\int_{\mathbb{S}^2}e^{-ikd\cdot y}\ov{\varphi(d)}ds(d),\qquad \text{for\;}y\in\R^3.
\enn
It then follows from the Green's theorem that
\ben
w_\ify(d)
= \int_{\pa D}\left(\frac{\pa e^{-ikd\cdot y}}{\pa\nu(y)}w(y)-e^{-ikd\cdot y}\frac{\pa w(y)}{\pa\nu(y)}\right)ds(y),
\enn
which combines with {\color{hw}the definition of} the incident field $\widetilde{w}^i$ further implies that, for $\varphi\in L^2(\mathbb{S}^2)$,
\be\label{3.2}
(G(f_1,f_2)^T, \varphi)_{L^2(\mathbb{S}^2)}
=\int_{\pa D}\left(\frac{\pa \widetilde{w}^i}{\pa\nu}w-\widetilde{w}^i\frac{\pa w}{\pa\nu}\right)ds.
\en
Notice that
\ben
\int_{\pa D}\left(\frac{\pa \widetilde{w}^s}{\pa\nu}w-\widetilde{w}^s\frac{\pa w}{\pa\nu}\right)ds=0,
\enn
where $\widetilde{w}^s:=\widetilde{w}-\widetilde{w}^i$ is the scattered field of the problem (\ref{1.1}). We then derive that
\ben
\int_{\pa D}\left(\frac{\pa \widetilde{w}^i}{\pa\nu}w-\widetilde{w}^i\frac{\pa w}{\pa\nu}\right)ds
&=&\int_{\pa D}\left(\frac{\pa \widetilde{w}}{\pa\nu}w -\widetilde{w}\frac{\pa w}{\pa\nu}\right)ds\\
&=&\int_{D_1}k^2(n_1-1)f_1\widetilde{w} dx+\int_{D_2}k^2(n_2-1)f_2\widetilde{w}dx.
\enn
This together with (\ref{3.2}) yields that
\be\label{3.3}
G^*\varphi = \left(k^2({\color{hw}\ov{n_1}}-1)\ov{\widetilde{w}_1}, k^2({\color{hw}\ov{n_2}}-1)\ov{\widetilde{w}_2}\right)^T
\en
with $\widetilde{w}_j:=\widetilde{w}|_{D_j}, j=1,2.$
Let now $G^*\varphi=0$, then $\widetilde{w}_1=0$ in $D_1$, $\widetilde{w}_2=0$ in $D_2$, which further implies $\widetilde{w}|_{\pa D}=\frac{\pa \widetilde{w}}{\pa\nu}|_{\pa D}=0$, this together with Holmgren's uniqueness theorem gives $\widetilde{w}=\widetilde{w}^i+\widetilde{w}^s=0$ in $\R^3\se \ov{D}$. Since
$\widetilde{w}^i$ does not satisfy the radiation condition {\color{hw}if $\varphi\neq0$}, {\color{hw}we obtain that $\widetilde{w}^i=0$ in $\R^3\se \ov{D}$}.
It then follows from
Theorem {\color{hw}3.19} in \cite{CK2013} that $\varphi=0$, and thus $G^*$ is injective, which proves the lemma.
\end{proof}
Introduce the far-field operator $F: L^2(\mathbb{S}^2)\mapsto L^2(\mathbb{S}^2)$ by
\be\label{3.4}
(Fg)(\wi{x}) = \int_{\mathbb{S}^2}u_\ify(\wi{x};d)g(d)ds(d)\qquad\text{for\;}g\in L^2(\mathbb{S}^2),
\en
where $u_\ify$ is the far-field pattern of the scattered field $u^s$ of the problem (\ref{1.1})-(\ref{1.2}) with the incident wave $u^i=e^{ikx\cdot d}$. Define the incident operator $H: L^2(\mathbb{S}^2)\mapsto Y$ by $H=(H_1,H_2)^T$ with
\be\label{3.5}
&&(H_1g)(x)=\int_{\mathbb{S}^2}e^{ikx\cdot d }g(d)ds(d)\qquad \text{for\;}x\in D_1,\\ \label{3.6}
&&(H_2g)(x)=\int_{\mathbb{S}^2}e^{ikx\cdot d }g(d)ds(d)\qquad \text{for\;}x\in D_2.
\en
It then follows from the superposition principle and the definition of the operator $G$ that $F=GH$. In order to {\color{hw}derive} the factorization of the far-field operator $F$, we next introduce the operator $\textbf{V}_j$ {\color{hw}that defined as follows: for $\varphi_j\in L^2(D_j)$,}
\be\label{3.8}
&&(\textbf{V}_j\varphi_j)(x) = \int_{D_j}\Phi(x,y)\varphi_j(y)dy\qquad \text{for\;}x\in D_j, \ j=1,2
\en
and the restriction operators $\textbf{V}_j^{(m)}:=\textbf{V}_j|_{D_m}$ $(j,m=1,2)$.
{\color{hw}Here, $\Phi(x,y)=\frac{e^{ik|x-y|}}{4\pi|x-y|}$ is the fundamental solution of Helmholtz equation $\tr u+k^2u=0$ in $\R^3 \se \{y\}$.}
Then we define the operator $T:Y\mapsto Y$ by
\be\label{3.70}
T=\left(
    \begin{array}{cc}
      q_1{\color{hw}I_{D_1}}-\textbf{V}_1^{(1)} & -\textbf{V}_2^{(1)}\\
     -\textbf{V}_1^{(2)} & q_2{\color{hw}I_{D_2}}-\textbf{V}_2^{(2)} \\
    \end{array}
  \right),
\en
where {\color{hw}for $j=1,2$, $q_j:=\frac{1}{k^2(n_j-1)}$ and ${I_{D_j}}$ are the identity operators on $L^2(D_j)$}.

{\color{hw}We have the following lemma on the property of the operator $T$.}

\begin{lemma}\label{le3.12}
The operator $T$ is invertible and
\be\label{3.16}
T^{-1}=T_1^{-1}+T_{com}
\en
{\color{hw}with $T_1$ to be an  invertible operator given by
\ben
T_1=\left(\begin{array}{cc}
                 {q_1}{\color{hw}I_{D_1}}&0 \\
                 0 & {q_2}{\color{hw}I_{D_2}} \\
                   \end{array}\right)
\enn
and the compact part $T_{com}=-T_1^{-1}T_2T^{-1}$ where $T_2$ is a compact operator}.
\end{lemma}
\begin{proof}
It is easily checked that the operator $T$ defined by (\ref{3.70}) can be divided into two parts
\ben
T&=&\left(\begin{array}{cc}
                  q_1{\color{hw}I_{D_1}}&0 \\
                 0 & q_2{\color{hw}I_{D_2}} \\
                   \end{array}\right)
-\left(\begin{array}{cc}
                \textbf{V}_1^{(1)} & \textbf{V}_2^{(1)} \\
                \textbf{V}_1^{(2)}& \textbf{V}_2^{(2)} \\
                  \end{array}\right)\\
&=:&T_1+T_2.
\enn
Clearly, $T_1$ is invertible on $Y$ and $T_2$ is compact on $Y$. This yields that the operator $T$ is of Fredholm-type with index $0$. Now let $T\varphi=0$ for $\varphi=(\varphi_1, \varphi_2)^{\color{hw}T}\in Y$. We define the function
\be\label{3.10}
w(x)
=\int_{D_1}\Phi(x,y)\varphi_1(y)dy+\int_{D_2}\Phi(x,y)\varphi_2(y)dy
\quad\text{for\;}x\in\R^3.
\en
{\color{hw}Then it follows from the properties of the operator $V_j,j=1,2$ (see e.g. \cite[Section 8.2]{CK2013}) that} $w$ is a solution of the problem (\ref{3.1}) with the data $f_1=f_2=0$. {\color{hw}Thus}, the uniqueness of (\ref{3.1}) ensures that $w=0$ in $\R^3$. So, {\color{hw}using the properties of the operator $V_j$ $(j=1,2)$ again,} we derive that
{\color{hw}$\tr w+k^2w=-\varphi_1$ in $D_1$ and $\tr w+k^2w=-\varphi_2$ in $D_2$ which yields} that $\varphi_1=\varphi_2=0$. {\color{hw}Therefore}, the invertibility of $T$ follows from the Fredholm alternative. {\color{hw}Finally}, by a direct calculation
{\color{hw} and the compactness of $T_2$}, one can obtain the assertion (\ref{3.16}) and {\color{hw} the compactness of $T_{com}$}. This ends the proof of the lemma.
\end{proof}

{\color{hw}Now, we present the factorization of the far-field operator $F$.}

\begin{theorem}\label{thm3.2}
Let the far-field operator $F$ be defined by (\ref{3.4}). Then we have the following factorization
\be\label{3.7}
F=H^*T^{-1}H,
\en
where $H^*$ is the adjoint operator of the incident operator $H$.
\end{theorem}
\begin{proof}
It can be easily proved that the adjoint $H^*$ of $H$ satisfies:
{\color{hw} for $\varphi=(\varphi_1, \varphi_2)^{T}\in Y$,}
\ben
(H^*\varphi)(d)
=\int_{D_1}e^{-ikd\cdot y}\varphi_1(y)dy+\int_{D_2}e^{-ikd\cdot y}\varphi_2(y)dy\quad
\text{for\;}d\in \mathbb{S}^2,
\enn
which is the far-field pattern of the function $w$ defined by (\ref{3.10}).
Then we derive that $w$ solves the problem (\ref{3.1}) with the data
\be\label{3.11}
f_1:=q_1\varphi_1-[\textbf{V}_1^{(1)}\varphi_1+\textbf{V}_2^{(1)}\varphi_2],\qquad
f_2:=q_2\varphi_2-[\textbf{V}_1^{(2)}\varphi_1+\textbf{V}_2^{(2)}\varphi_2].
\en
So that $H^*=GT$, which together with Lemma \ref{le3.12} yields $G=H^*T^{-1}$. This combines with the fact $F=GH$ leads to the factorization that $F=H^*T^{-1}H$, which completes the proof of the Theorem.
\end{proof}
We now introduce an auxiliary operator $\wid{H}_1:L^2(\mathbb{S}^2)\mapsto H^{\frac{1}{2}}(\pa D_1)$ defined by
\be\label{3.12}
(\wid{H}_1\varphi)(x) = \int_{\mathbb{S}^2}e^{ikx\cdot d}\varphi(d)ds(d),\qquad \text{for\;} x\in\pa D_1
\en
and the compact operator
$L: H^{\frac{1}{2}}(\pa D_1)\mapsto L^2(D_1)$
with $Lh=w|_{D_1}$, where $w$ is a solution of {\color{hw} the problem}
\be\label{3.120}
\left\{
  \begin{array}{ll}
    \tr w + k^2w = 0, & \text{in}\;D_1,\\
    w= h, & \text{on}\; \pa D_1,\\
  \end{array}
\right.
\en
for $h\in H^{\frac{1}{2}}(\pa D_1)$.
In the following, we always assume that $k^2$ is not a Dirichlet eigenvalue of $-\tr$ in $D_1$. So, the above problem (\ref{3.120}) is well posed and consequently, the operator $L$ is well defined.

It is noted that $L\widetilde{H}_1=H_1$ and thus
\be\label{3.13}
H=\left(\begin{array}{cc}
                  L&0 \\
                 0 & I_{D_2} \\
                   \end{array}\right)
\left(
  \begin{array}{c}
    \widetilde{H}_1 \\
    H_2 \\
  \end{array}
\right):=\emph{\textbf{A}}\wid{H}.
\en
Based on (\ref{3.13}), we define a series of perturbation operators $F_m$ by
\be\label{2.140}
F_m:=F+\rho_m\wid{H}_1^*N_{i,\pa D_1}\wid{H}_1
\en
with $\rho_m>0$ for every $m\in\N$, which satisfies that $\rho_m\to 0$ as $m\to\ify$. Here, $N_{i,\pa D_1}$ is defined by
\ben
&&(N_{i,\pa D_1}\varphi_1)(x)=\frac{\pa}{\pa\nu(x)}\int_{\pa D_1}\frac{\pa\Phi(i;x,y)}{\pa\nu(y)}\varphi_1(y)ds(y),
                           \qquad \text{for\;}x\in \pa D_1,
\enn
where $\Phi(i;x,y)$ is a fundamental solution of the special Helmholtz equation $\tr u-u=0$.
Notice that $\widetilde{H}_1$ is well-defined, it then follows
\be\label{3.14}
\|F_m-F\|_{L^2(\mathbb{S}^2)} = \|\rho_m\wid{H}_1^*N_{i,\pa D_1}\wid{H}_1\|_{L^2(\mathbb{S}^2)}
{\color{hw}= \rho_m\|\wid{H}_1^*N_{i,\pa D_1}\wid{H}_1\|_{L^2(\mathbb{S}^2)}}\to 0,\;\text{as\;}m\to\ify.
\en
We define an auxiliary matrix $J_m$ by
\ben
J_m:=\left(\begin{array}{cc}
                  \rho_m N_{i,\pa D_1}&0 \\
                 0 & 0 \\
                   \end{array}\right).
\enn
This, together with (\ref{3.16}) leads to the factorization for the perturbation far-field operator $F_m$ as
\be\no
F_m
&=&\wid{H}^*\left(\emph{\textbf{A}}^*T^{-1}\emph{\textbf{A}}+J_m\right)\wid{H}\\\no
&=&\wid{H}^*\left[\left(\begin{array}{cc}
                  \rho_m N_{i,\pa D_1} & 0 \\
                  0 & \frac{1}{q_2}{\color{hw}I_{D_2}} \\
                   \end{array}\right)
+\left[\left(\begin{array}{cc}
                  L^*{\color{hw}(\frac{1}{q_1}I_{D_1})}L & 0 \\
                  0 & 0 \\
                   \end{array}\right)
             +\emph{\textbf{A}}^*T_{com}\emph{\textbf{A}}
\right]\right]\wid{H}\\\label{3.17}
&=:&\wid{H}^*\left(M_m+M_{com}\right)\wid{H}.
\en
It is obvious that $M_{com}$ is compact on $\wid{Y}:=H^{\frac{1}{2}}(\pa D_1)\times L^2(D_2)$ and {\color{hw}$-\Rt M_m$} is coercive on $\wid{Y}$, {\color{hw}i.e., there exists
$c_0>0$ with $-\langle\Rt M_m\varphi,\varphi\rangle\geq c_0\|\varphi\|^2$ for all $\varphi\in\wid{Y}$}
because $\Rt(n_2)-1<-c$ in $D_2$ for some positive constant $c$ and the operator {\color{hw}$-N_{i,\pa D_1}$} is coercive {\color{hw}(see e.g. \cite[Theorem 1.26]{K2008})}.

For $z\in \R^3$,  define the function $\phi_z(\wi{x})=e^{-ik\wi{x}\cdot z}$ with $\wi{x}\in \mathbb{S}^2$.
Then we next prove the fact that $z\in D\quad\Longleftrightarrow\quad \phi_z\in R(\wid{H}^*)$. {\color{hw}To this end,} we first need to show the following lemma.
\begin{lemma}\label{le3.2}
For $z\in \R^3$,  we have that
\ben
z\in D\quad\Longleftrightarrow\quad \phi_z\in  R(G).
\enn
\end{lemma}
\begin{proof}
Assume first $z\in D$ and thus there exists a closed ball $B_\delta(z)$ centered at $z$ with
radius $\delta>0$ such that $B_\delta(z)\subset D$. Then we choose a cut-off function
$\chi\in C^\ify(\R^3)$ with $\chi(t)=1$ for $|t|\geq\delta$ and $\chi(t)=0$ for
$|t|\leq\delta/2$ and define a function $w(x)$ by
\ben
w(x)
:=\chi(|x-z|)\Phi(x,z)
=\chi(|x-z|)\frac{e^{ik|x-z|}}{4\pi|x-z|}\qquad\text{in\;}\R^3.
\enn
Obviously, $w\in C^\ify(\R^3)$ and $w(x)=\Phi(x,z)$ for $|x-z|\geq\delta$. Indeed, for $x\in D_j$ $(j=1, 2)$, we have
\ben
\tr w+k^2n_jw
= \Phi\tr\chi+\chi\tr\Phi+2\nabla\chi\cdot\nabla\Phi+k^2n_j\chi\Phi
=: {\color{hw}k^2(1-n_j)}g_j\quad\text{in\;}D_j.
\enn
Then $g_j\in L^2(D_j)$, this combines with the unique solvability of the problem (\ref{3.1}) implies that
$w$ is the solution of (\ref{3.1}) with the data $(f_1,f_2)^T=(g_1,g_2)^T$.
So we immediately get $G(g_1,g_2)^T=w_\ify=\phi_z$ and consequently  $\phi_z\in R(G)$.

On the other hand, let $z\notin D$ and assume that there exists $(\wid{f}_1,\wid{f}_2)^T\in Y$
such that $G(\wid{f}_1,\wid{f}_2)^T=\phi_z$. Let $\wid{w}$ be the solution of the problem (\ref{3.1}) with
the data $(f_1,f_2)^T=(\wid{f}_1,\wid{f}_2)^T$
and $\wid{w}_{\ify}$ be the far-field pattern of $\wid{w}$. Then $\wid{w}_{\ify}=\phi_z$.
{\color{hw}It follows from Rellich's Lemma and unique continuation theorem that} $\wid{w}(x)=\Phi(x,z)$ in $\R^3\se(\ov{D}\cup\{z\})$. However, this is a contradiction
because $\|\wid{w}\|_{H^1(B_\delta(z))}<\ify$ and $\|\Phi(\cdot,z)\|_{H^1(B_\delta(z))}=\ify$, {\color{hw}where $B_\delta(z)$ is chosen
to be a sufficiently small ball centered at $z$}. The proof of this lemma is thus completed.
\end{proof}
It is noted that, from the proof of Theorem \ref{thm3.2}, the solution operator $G$ and the incident operator $H$ satisfy $H^*=GT$ and $G=H^*T^{-1}$. This implies that $R(H^*)=R(G)$, thus we have the following lemma. The proof is easily obtained, hence we omit it.
\begin{lemma}\label{le3.4}
It holds that
\ben
z\in D\quad\Longleftrightarrow\quad \phi_z\in  R(H^*).
\enn
\end{lemma}
Combining the above lemmas yields the following theorem.
\begin{theorem}\label{thm3.5}
$\wid{H}^*$ is compact with dense range in $L^2(\mathbb{S}^2)$ and
\ben
z\in D\quad\Longleftrightarrow\quad \phi_z\in R(\wid{H}^*).
\enn
\end{theorem}
\begin{proof}
{\color{hw}From \cite[Theorem 3.19]{CK2013}, we can easily obtain that $\wid{H}$ is injective.}
Then the compactness and denseness of $\wid{H}^*$ easily follow from the compactness and injectivity
of $\wid{H}$. Assume $z\in D$, it is seen from Lemma \ref{le3.4} that $\phi_z\in R(H^*)$. {\color{hw}Let ${Y}'$ denote the {\color{hw} adjoint} of ${Y}$.}
So there exists $\varphi\in Y'$ such that $\phi_z=H^*\varphi$. Notice that $H^*=\wid{H}^*A^*$ from (\ref{3.13}) and thus $\phi_z=\wid{H}^*(\emph{\textbf{A}}^*\varphi)$, which gives
$\phi_z\in R(\wid{H}^*)$.

{\color{hw}On the other hand, let $z\notin D$ and assume on the contrary that there exists $\varphi=(\varphi_1,\varphi_2)^T\in \wid{Y}'$
such that $\wid{H}^*\varphi=\phi_z$.} Here, $\wid{Y}'$ denotes the {\color{hw} adjoint} of $\wid{Y}$. Then {\color{hw}from Rellich's Lemma and unique continuation theorem}, it can be obtained that
\be\label{3.18}
\int_{\pa D_1}\Phi(\cdot,y){\color{hw}\varphi_1(y)}dy+\int_{D_2}\Phi(\cdot,y){\color{hw}\varphi_2(y)}dy=\Phi(\cdot,z)
\quad \text{in\;}\R^3\se(\ov{D}\cup\{z\}).
\en
However, this is a contradiction because the left-hand of (\ref{3.18}) belongs to {\color{hw}$H^1(B_\delta(z))$} but the right-hand {\color{hw} of (\ref{3.18})} does not belong to {\color{hw}$H^1(B_\delta(z))$}, {\color{hw}where $B_\delta(z)$ is chosen
to be a sufficiently small ball centered at $z$}. This proves the theorem.
\end{proof}

{\color{hw}To proceed further we need to introduce the following interior transmission eigenvalue.}

\begin{definition}${}^{\cite[Definition 4.7]{K2008}}$\label{def3.6}
$k^2$ is called an interior transmission eigenvalue if there exists $(u,w)\in H_0^1(D)\times L^2(D)$ with $(u,w)\neq(0,0)$ and a sequence ${w_j}\in H^2(D)$ with $w_j\rightarrow w$ in $L^2(D)$ and $\tr w_j+k^2w_j=0$ in $D$ {\color{hw} and $(u,w)$} satisfies
\be\label{3.19}
\int_{D}\left(\nabla u\cdot\nabla\varphi-k^2 {\color{hw}n}u\varphi\right)dx
=k^2\int_{D}{\color{hw}(n-1)}w\varphi dx \quad\text{for\;all\;}\varphi\in H^1(D).
\en
\end{definition}

It should be remarked that the eigenvalue problem (\ref{3.19}) has been studied by Kirsch
in \cite{K2008} {\color{hw} when $n$ is real-valued
and  $n(x)>1$ or $n(x)<1$ in $D$, where it was proved that} (\ref{3.19}) has at most a
countable number of eigenvalues $k^2>0$. In the current paper, we always assume that $k^2>0$ is not an interior transmission eigenvalue {\color{hw}under} the case when $\Rt[n_1(x)]-1\geq c$ in $D_1$ and $\Rt[n_2(x)]-1<-c$ in $D_2$ for some positive
constant $c$.

In order to show the main theorem of the factorization method for our perturbation far-field pattern $F_m$ {\color{hw}that derived in (\ref{3.17})}, we need to prove the following theorem.

\begin{theorem}\label{thm3.7}
  Let $\wid{M}_m=M_m+M_{com}$ and assume that $k^2>0$ is neither an interior transmission eigenvalue in the sense of Definition \ref{def3.6} nor a Dirichlet eigenvalue of $-\tr$ in $D_1$.
Then\\
(i) $\Rt\wid{M}_m=\wid{M}_m^{(1)}+\wid{M}_m^{(2)}$,
where {\color{hw}$-\wid{M}_m^{(1)}$} is coercive and $\wid{M}_m^{(2)}$ is a compact operator;\\
(ii) $\I \wid{M}_m$ is strictly positive on $\ov{R(\wid{H})}$, i.e., $\I\langle\wid{M}_m\varphi,\varphi\rangle>0$ for all $\varphi\in\ov{R(\wid{H})}$ with $\varphi\not=0$.
\end{theorem}
\begin{proof}
(i) The assertion (i) follows easily from the properties of
$M_m$ and $M_{com}$ (see the sentences under (\ref{3.17})).\\
(ii) Since $N_{i,\pa D_1}$ is self-adjoint, we obtain $\langle J_m\varphi, \varphi\rangle=0$ for all $\varphi\in\ov{R(\wid{H})}$ and
\ben
\I\langle\wid{M}_m\varphi,\varphi\rangle
&=&\I\langle \left(\emph{\textbf{A}}^*T^{-1}\emph{\textbf{A}}+J_m\right)\varphi,\varphi\rangle
=\I\langle T^{-1}\emph{\textbf{A}}\varphi,\emph{\textbf{A}}\varphi\rangle\\
&=&\I\langle \emph{\textbf{A}}\varphi,T^{-1*}\emph{\textbf{A}}\varphi\rangle
=\I\langle T^*(T^{-1*}\emph{\textbf{A}}\varphi),(T^{-1*}\emph{\textbf{A}}\varphi)\rangle\\
&=& {\color{hw}\I\langle (T^{-1*}\emph{\textbf{A}}\varphi),T(T^{-1*}\emph{\textbf{A}}\varphi)\rangle
=-\I\langle T(T^{-1*}\emph{\textbf{A}}\varphi),(T^{-1*}\emph{\textbf{A}}\varphi)\rangle}
\enn
for all $\varphi\in\ov{R(\wid{H})}$. The fact $\varphi\in \ov{R(\wid{H})}$ implies that there exists ${\color{hw}g_p}\in L^2(\mathbb{S}^2)$ {\color{hw}for $p\in\N$}
such that $\wid{H}{\color{hw}g_p}\to\varphi$ as $p\to\ify$. With the aid of $A\wid{H}=H$, we derive
$T^{-1*}H {\color{hw}g_p}\to T^{-1*}\emph{\textbf{A}}\varphi$ as ${\color{hw}p}\to\ify$.
{\color{hw}Then using $H^*=G T$ that obtained in the proof of Theorem \ref{thm3.2}, we have} $T^{-1*}H {\color{hw}g_p}= G^*{\color{hw}g_p}$. Therefore, the assertion (ii) is equivalent to
\ben
{\color{hw}\I\langle T\psi,\psi\rangle<0}\qquad \text{for\;all\;}\psi\in\ov{R(G^*)}\;\text{with\;}\psi\not=0.
\enn

Firstly, we prove that
\be\label{3.190}
\I\langle T\psi,\psi\rangle\leq0\qquad \text{for\;all\;}\psi\in\ov{R(G^*)}\;\text{with\;}\psi\not=0.
\en
Define two functions $w_1$ and $w_2$ by
\ben
&&w_1(x):=\int_{D_1}\Phi(x,y)\psi_1(y)dy,\qquad\text{for\;}x\in \R^3,\\
&&w_2(x):=\int_{D_2}\Phi(x,y)\psi_2(y)dy,\qquad\text{for\;}x\in \R^3
\enn
and $w:=w_1+w_2$ with $\psi=(\psi_1,\psi_2)^{{\color{hw}T}}$, we then obtain
\be\no
&&{\color{hw}\langle}T (\psi_1, \psi_2)^T,(\psi_1, \psi_2)^T{\color{hw}\rangle}_{Y\times Y'}
={\color{hw}\langle} (f_1,f_2)^T,(\psi_1, \psi_2)^T{\color{hw}\rangle}_{Y\times Y'}\\\no
&=&(q_1\psi_1,\psi_1)_{D_1}-(w_1,\psi_1)_{D_1}-(w_2,\psi_1)_{D_1}
-(w_1,\psi_2)_{D_2}+(q_2\psi_2,\psi_2)_{D_2}-(w_2,\psi_2)_{D_2}\\ \label{3.20}
&=&:I_1+I_2+I_3+I_4+I_5+I_6,
\en
where $(f_1,f_2)^T\in Y'$ is defined by (\ref{3.11}) with $(\varphi_1,\varphi_2)^T$ replaced by $(\psi_1, \psi_2)^T$.
In view of the fact that $\I [n_1(x)]\geq 0$ and $\I [n_2(x)]\geq 0$, we can easily observe that $\I I_1\leq 0$ and $\I I_5\leq 0$. By the definition of the function $w_1$ and the Green's theorem we derive
\be\no
I_2&=&\int_{D_1} w_1\ov{(\Delta w_1+k^2w_1)}dx\\\no
&=& \int_{\pa D}\frac{\pa \ov{w}_1}{\pa \nu}w_1ds
   -\int_{D_1}\left(|\nabla w_1|^2-k^2|w_1|^2\right)dx\\\label{3.21}
&=&\int_{\pa B_R}\frac{\pa \ov{w}_1}{\pa r}w_1ds
   -\int_{B_R}\left(|\nabla w_1|^2-k^2|w_1|^2\right)dx,
\en
where $B_R\supset D$ is a ball centered at $0$ with radius $R$.
It then follows from the Sommerfeld radiation condition that
\be\label{3.22}
\I I_2
=\I\left(\lim_{R\to\ify}\int_{\pa B_R}\frac{\pa \ov{w}_1}{\pa r}w_1ds\right)
=-\frac{k}{(4\pi)^2}\int_{\mathbb{S}^2}|w_{1,\infty}|^2ds,
\en
where $w_{1,\infty}$ is the far-field pattern of $w_1$.
Similarly, we can show that
\be\label{3.23}
\I I_6
   =\I\left(\lim_{R\to\ify}\int_{\pa B_R}\frac{\pa \ov{w}_2}{\pa r}w_2ds\right)
=-\frac{k}{(4\pi)^2}\int_{\mathbb{S}^2}|w_{2,\infty}|^2ds
\en
as well as
\be\no
\I(I_3+I_4)
&=&\I\left[\lim_{R\to\infty}\left(\int_{\pa B_R}\frac{\pa \ov{w}_2}{\pa r}w_1
      -\frac{\pa \ov{w}_1}{\pa r}w_2ds\right)\right]\\\label{3.25}
&=&-\frac{2k}{(4\pi)^2}\Rt\left[\int_{\mathbb{S}^2}w_{1,\infty} \ov{w}_{2,\infty }ds\right],
\en
where $w_{2,\infty}$ is the far-field pattern of $w_2$. Denote by $w_{\infty}$ the far-field pattern of $w$.
Then the Cauchy-Schwarz inequality together with
(\ref{3.20})-(\ref{3.25}) yields
\be\label{eq1}
\I{\color{hw}\langle}T (\varphi_1, \varphi_2)^T,(\varphi_1, \varphi_2)^T{\color{hw}\rangle}_{Y\times Y'}
\leq -\frac{k}{(4\pi)^2}\|w_{\ify}\|^2_{L^2(\mathbb{S}^2)} \leq0,
\en
which proves the assertion (\ref{3.190}).

Secondly, let $\psi^{(0)}=(\psi_1^{(0)},\psi_2^{(0)})^{{\color{hw}T}}\in \ov{R(G^*)}$ such that $\I{\color{hw}\langle}T\psi^{(0)},\psi^{(0)}{\color{hw}\rangle}=0$.
We define
\ben
w^{(0)}(x)=\int_{D_1}\Phi(x,y){\color{hw}\ov{\psi_1^{(0)}(y)}}dy+\int_{D_2}\Phi(x,y){\color{hw}\ov{\psi_2^{(0)}(y)}}dy
\enn
and $w^{(0)}_\infty$ to be the far-field pattern of $w^{(0)}$.
Then {\color{hw}it follows from (\ref{eq1})} that $w^{(0)}_\infty=0$ and thus $w^{(0)}=0$ in $\R^3\se\ov{D}$
due to Rellich's Lemma.
Hence, {\color{hw}using this and the fact that $w^{(0)}\in H^2_{loc}(\R^3)$}, we obtain $w^{(0)}|_{D}\in H^1_0(D)$ {\color{hw}and ${\pa w^{(0)}}/{\pa\nu}=0$ on ${\pa D}$}. Since ${\color{hw}{\psi}^{(0)}}\in \ov{R(G^*)}$ there exists $\psi_j=G^*g_j$ such that $\psi_j\to {\color{hw}{\psi}^{(0)}}$ in $L^2(D)$ as $j\to\ify$. Further it follows from (\ref{3.3}) that
\be\label{3.26}
\left({\color{hw}\wid{q}_1}\wid{w}_j|_{D_1}, {\color{hw}\wid{q}}_2\wid{w}_j|_{D_2}\right)^{{\color{hw}T}}\to {\color{hw}\ov{\psi^{(0)}}}\quad{\color{hw}\text{in\;}L^2(D)}\;
\text{as\;}j\to\ify,
\en
where $\wid{q}_l=k^2(n_l-1), l =1,2$ and $\wid{w}_j$ is the total field of the problem (\ref{1.1})-(\ref{1.2}) corresponding to the incident field $\wid{w}^i_j{\color{hw}(x)}:=\int_{\mathbb{S}^2}e^{-ik {\color{hw}d\cdot x}}{\color{hw}\ov{{\color{black}g_j}({\color{hw}d})}}ds({\color{hw}d})$. We now define
\be\label{3.27}
w_j(x)=\int_{D_1}\Phi(x,y){\color{hw}\wid{q}_1(y)}\wid{w}_{1j}(y)dy
+\int_{D_2}\Phi(x,y){\color{hw}\wid{q}_2(y)}\wid{w}_{2j}(y)dy
\en
with $\wid{w}_{ij}:=\wid{w}_{j}|_{D_i}, i=1,2$. Then we have
$w_j\to w^{(0)}$ in $H^1(D)$ as $j\to\ify$.

It is noted that $\tr w_j+k^2w_j=-{\color{hw}\wid{q}}\wid{w}_j$ in $D$ with ${\color{hw}\wid{q}}:=\wid{q}_l$ in $D_l, l=1,2$. Then by the Green's representation
theorem, we derive
\ben
\wid{w}_j(x)&=&\int_{\pa D}\left\{
   \frac{\pa \wid{w}_j(y)}{\pa\nu(y)}\Phi(x,y)-\wid{w}_j(y)\frac{\pa\Phi(x,y)}{\pa\nu(y)}\right\}ds(y)\\
   &&+\int_{D}\Phi(x,y){\color{hw}\wid{q}}(y)\wid{w}_j(y)dy
   \quad \text{in\;} D.
\enn
This in combination with the definition of $w_j$ in (\ref{3.27}) yields
\be\no
\wid{w}_j(x)-w_j(x)
&=&\int_{\pa D}\left\{\frac{\pa \wid{w}_j(y)}{\pa\nu(y)}\Phi(x,y)
             -\wid{w}_j(y)\frac{\pa\Phi(x,y)}{\pa\nu(y)}\right\}ds(y)\\
&=&\int_{\mathbb{S}^2}e^{-ikd\cdot x}\ov{g_j(d)}ds(d):=v_j(x)\qquad \text{in\;} D.
\en
Then we conclude that
\be\label{3.28}
v_j\to\frac{1}{{\color{hw}\wid{q}}}{\color{hw}\ov{\psi^{(0)}}}-w^{(0)}\qquad \text{in\;} L^2(D)
\en
due to the fact that $\wid{w}_j\to\frac{1}{{\color{hw}\wid{q}}}{\color{hw}\ov{\psi^{(0)}}} $ and $w_j\to w^{(0)}$ in $L^2(D)$. Setting $v:=\frac{1}{{\color{hw}\wid{q}}}{\color{hw}\ov{\psi^{(0)}}}-w^{(0)}$, then {\color{hw}$(w^{(0)},v)\in H_0^1(D)\times L^2(D)$}, $v_j\in H^2(D)$ satisfying $\tr v_j+k^2v_j=0$ in $D$, {\color{hw}$v_j\rightarrow v$ in $L^2(D)$} and ${\color{hw}w^{(0)}}|_D$ solves the following problem
\ben
\tr w^{(0)}+k^2nw^{(0)}=-{\color{hw}\wid{q}}v\qquad \text{in\;} D.
\enn
{\color{hw}From this and the fact that ${\pa w^{(0)}}/{\pa\nu}=0$ on ${\pa D}$, it follows that $(w^{(0)},v)$ satisfy
(\ref{3.19}) with $u$ and $w$ replaced by $w^{(0)}$ and $v$, respectively.}
Since $k^2>0$ is not an interior transmission eigenvalue in the sense of Definition \ref{def3.6}, it then follows that $(w^{(0)},v)$ has to vanish {\color{hw}in $D$}. Thus $\psi^{(0)}=0$ due to the fact that $\tr w^{(0)}+k^2w^{(0)}=-{\color{hw}\ov{\psi^{(0)}}}$ in $D$. {\color{hw}This, together with the assertion (\ref{3.190}), proves the statement (ii) of this lemma.}
\end{proof}
Finally, the Range Identity in \cite[Theorem 2.15]{K2008} in combination with Theorems \ref{thm3.2}, \ref{thm3.5} and \ref{thm3.7} gives the following main theorem in this section.
\begin{theorem}\label{thm3.8}
Assume that the conditions presented in Theorem \ref{thm3.7} hold true. Then
\ben
   z\in D \quad&\Longleftrightarrow&\quad \phi_z\in R(F_{m,\#}^{\frac{1}{2}})\\
            &\Longleftrightarrow&\quad
             W_m(z):=\left[\sum_j\frac{|\langle\phi_z,\psi^{(m)}_j \rangle_{L^2(\mathbb{S}^2)}|^2}{\la^{(m)}_j}\right]^{-1}>0
\enn
for every fixed $m\in\N$,
where $\{\la^{(m)}_j;\psi^{(m)}_j\}_{j\in\N}$ is an eigen-system of the self-adjoint operator
$F_{m,\#}:=|\Rt F_m|+|\I F_m|$.
\end{theorem}

\begin{remark}\label{rm1}{\rm
Since the classical factorization method can not be directly applied to deal with our inverse problem associated with the complex refractive index. Then we instead construct a sequence of perturbed operators $F_m$ by (\ref{2.140}) of the far-field operator $F$. It is shown in this section that $F_m$ has a factorization satisfying the Range Identity in \cite[Theorem 2.15]{K2008} for every $m\in\N$. Consequently, the support of the inhomogeneous medium $D$ can be recovered from the spectral data of $F_{m, \#}$ for every $m\in\N$. We point out that,  {\color{hw}due to (\ref{3.14}),} if $m_0$ is sufficiently large then the exact operator $F_\#$ can be regarded as a sufficiently small perturbation of $F_{m_0,\#}$ and the noisy operator $F^\delta_\#$ with the noise level $\delta$ of $F_\#$ can also be regarded as a
sufficiently small perturbation of $F^\delta_{m_0,\#}$ with the noise level $\delta$. Based on the above discussions, in the numerical examples presented in the next section, we just use the spectral data of $F$ and $F^\delta$ to numerically reconstruct the shape and location of $D$. }
\end{remark}

\begin{remark}{\color{hw}
Theorem \ref{thm3.8} remains true for the
two-dimensional case. The proof is similar with
minor modifications.}
\end{remark}

\section{Numerical examples}
\setcounter{equation}{0}

In this section, numerical experiments in two dimensions are carried out to demonstrate the efficiency of the approximate factorization method.
To generate the synthetic far-field data, we make use of
the finite element method on a truncated domain
enclosed by a PML layer with uniform meshes (see e.g. \cite{CW08} for the PML technique).
Further, the far-field data $u_\infty(\wi{x};d)$
are discretized for a finite number of observation directions $\wi{x}_r\in \mathbb{S}^1$ and incident directions
$d_s\in \mathbb{S}^1$ with $r,s=1,2,\ldots,M$, which are equidistantly
distributed on the unit circle $\mathbb{S}^1$. Thus the measured data are obtained as the
matrix $F_M=(u_\infty(\wi{x}_r;d_s))_{1\leq r,s \leq M}\in \mathbb{C}^{M\times M}$. Then the indicator function
$W(z)$ for the far-field operator $F$ is approximated as follows:
\be\label{5.1}
W_M(z)=\left[\sum^M_{p=1}\frac{1}{\lambda_p}
\left|\sum^M\limits_{q=1}
\phi_{z,q}\overline{\psi_{p,q}}\right|^2\right]^{-1}
\quad \textrm{for\;}z\in{\mathbb R}^2,
\en
where $\{\phi_{z,q}\}^{M}_{q=1}$ is the discretization of
the test function $\phi_z$
and $\{\lambda_p;\psi_p\}^M_{p=1}$ is
the eigen-system of the self-adjoint matrix
$F_{M,\#}:=|\textrm{Re}(F_M)|+|\textrm{Im}(F_M)|$
with $\psi_p=(\psi_{p,q})^M_{q=1}$.
From Theorem \ref{thm3.8}, it is expected that $W_M(z)$ is
much bigger for $z\in D$ than that for $z\notin D$.

In each examples, we will also show the reconstructed results for the  approximate factorization method from noisy data. For the noisy data,
a complex-valued noise matrix $X$ is added to the
data matrix $F_M$, where $X=(x_{rs})_{1\leq r,s\leq M}$ with $x_{rs}=\xi_{rs}+i\zeta_{rs}$ and
$\xi_{rs},\zeta_{rs}$
are normally distributed random numbers in $[-1,1]$.
Then the perturbed matrix with noisy level $\delta>0$
can simulated as follows:
\be
&&F^\delta_M:=F_M+\delta\frac{X}{\|X\|_2}\|F_M\|_2,\\
&&(F^\delta_M)_\#:=|\textrm{Re}(F^\delta_M)|+|\textrm{Im}(F^\delta_M)|.
\en
Accordingly, the truncated indicator function $W_M(z)$
can be computed from the eigen-system of the
perturbed matrix $(F^\delta_M)_\#$ which is similar as (\ref{5.1}).

In the following examples, we set $M=64$, $k=5$ and the test curves for the boundary $\pa D$ are given in Table \ref{t1}.
The indicator function $W_N(z)$ is plotted against the sampling
point $z\in\mathbb{R}^2$.

\begin{table}
\centering
\begin{tabular}{ll}
\hline
\textbf{Curve type}& \textbf{Parametrization:}\\
\hline
Kite shaped & $x(t)=(\cos{t}+0.65\cos{2t}-0.65,1.5\sin{t})$, $t\in[0,2\pi]$\\
Rounded square & $x(t)=$$({1}/{2})(\cos^3{t}+\cos{t},\sin^3{t}+\sin{t})$, $t\in[0,2\pi]$\\
Rounded triangle & $x(t)=(2+0.3\cos{3t})(\cos{t},\sin{t})$, $t\in[0,2\pi]$\\
\hline
\end{tabular}
\caption{Parametrization of the curve}\label{t1}
\end{table}

\textbf{Example 1.} In this example, we consider the case when
$\pa D$ is a rounded triangle-shaped boundary and the  refractive index $n$ in $D$ is given by
\ben
n(x)=\left\{
  \begin{array}{ll}
    2+2i & \text{for}\;x\in D_1,\\
    0.5+2i & \text{for}\; x\in D_2,
  \end{array}
\right.
\enn
where $D_1=\{(x_1,x_2)\in D: x_2>0\}$ and $D_2=\{(x_1,x_2)\in D: x_2<0\}$. See Figure \ref{fig1}(a) for
the physical configuration.
The reconstruction results of the boundary $\pa D$ are presented in Figure \ref{fig1}
by using the far-field data without noise, with 5\% noise
and with 10\% noise, respectively.

\begin{figure}[htbp]
\begin{minipage}[t]{0.4\linewidth}
\centering
\includegraphics[width=3in]{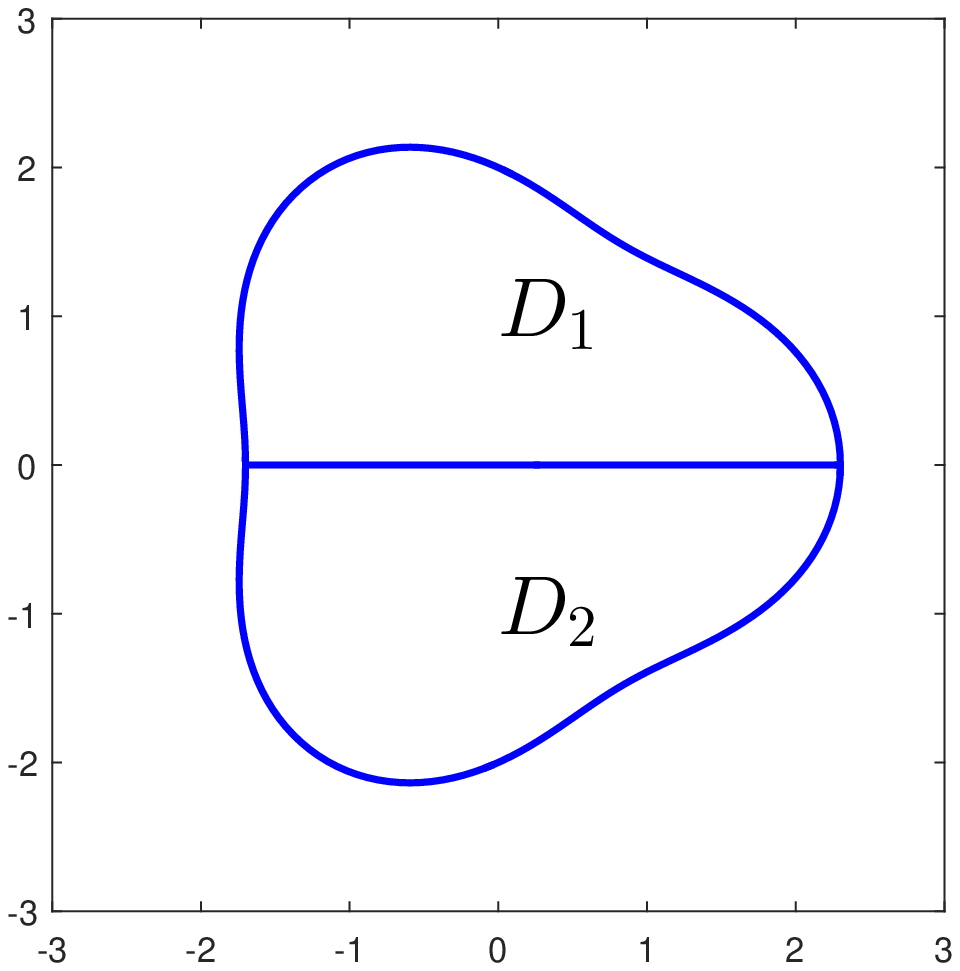}
(a) Physical configuration
\end{minipage}\qquad\qquad
\begin{minipage}[t]{0.4\linewidth}
\centering
\includegraphics[width=3in]{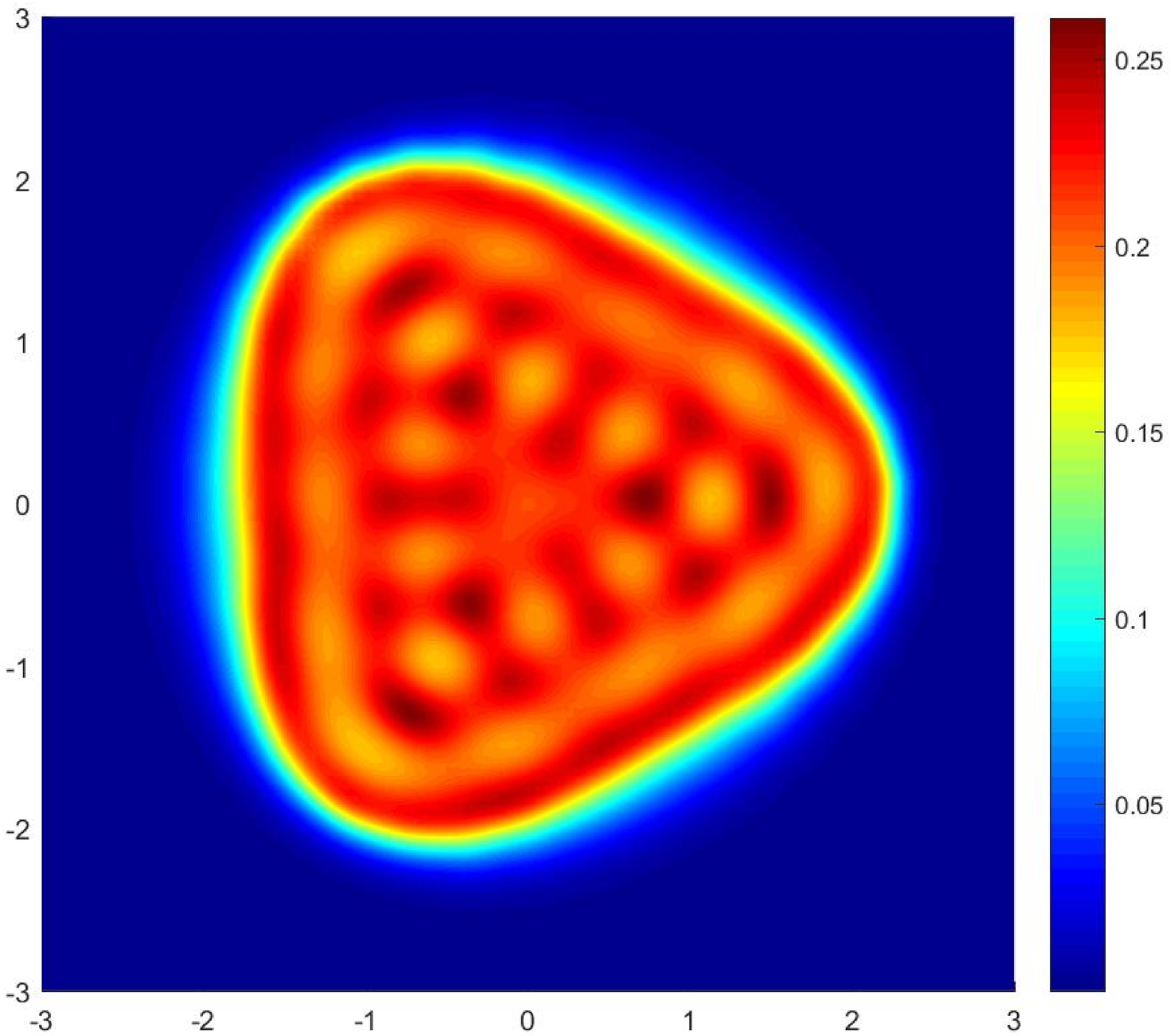}
(b) k=5, no noise
\end{minipage}
%
\begin{minipage}[t]{0.4\linewidth}
\centering
\includegraphics[width=3in]{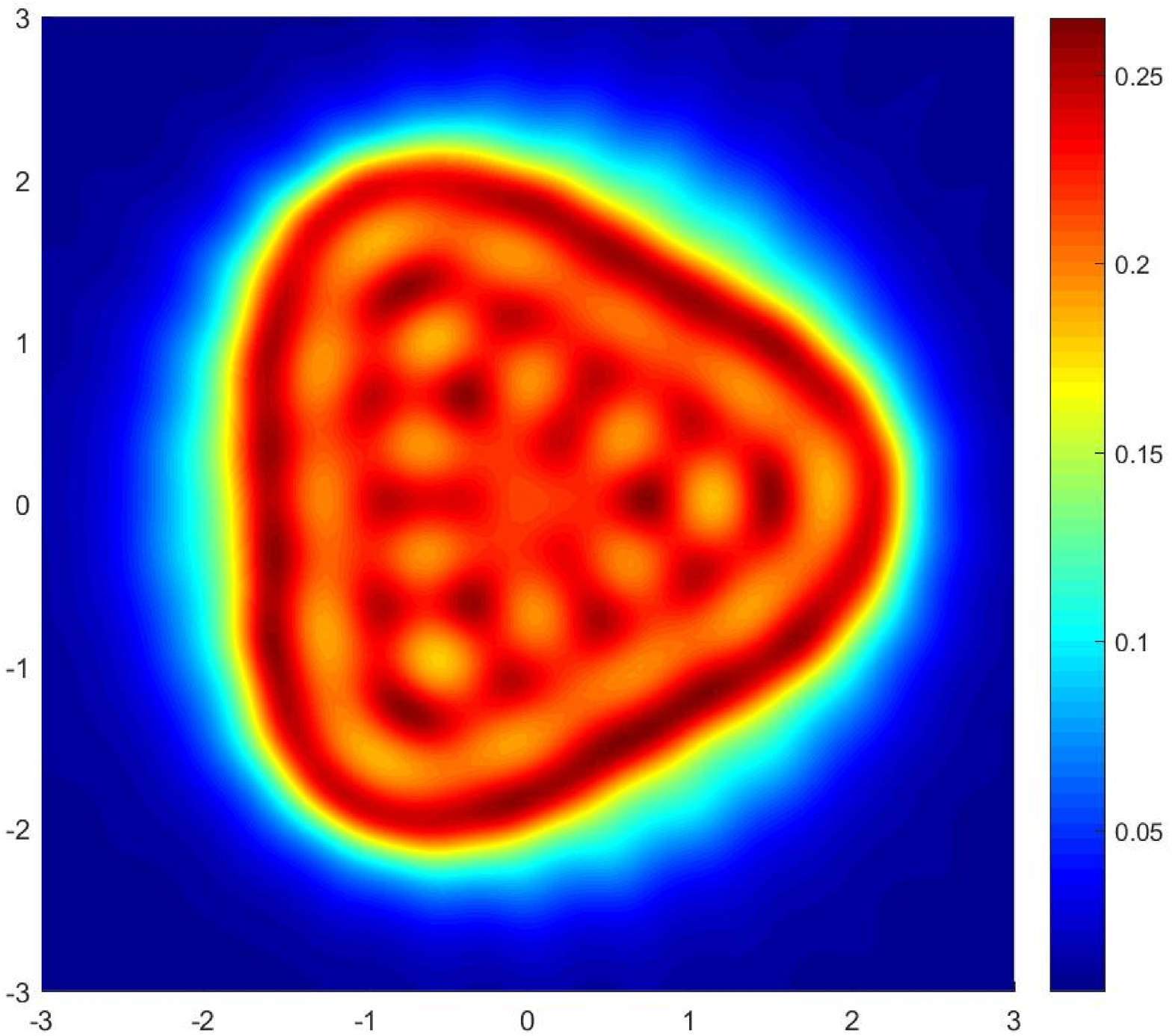}
(c) k=5, 5\% noise
\end{minipage}\qquad\qquad
\begin{minipage}[t]{0.4\linewidth}
\centering
\includegraphics[width=3in]{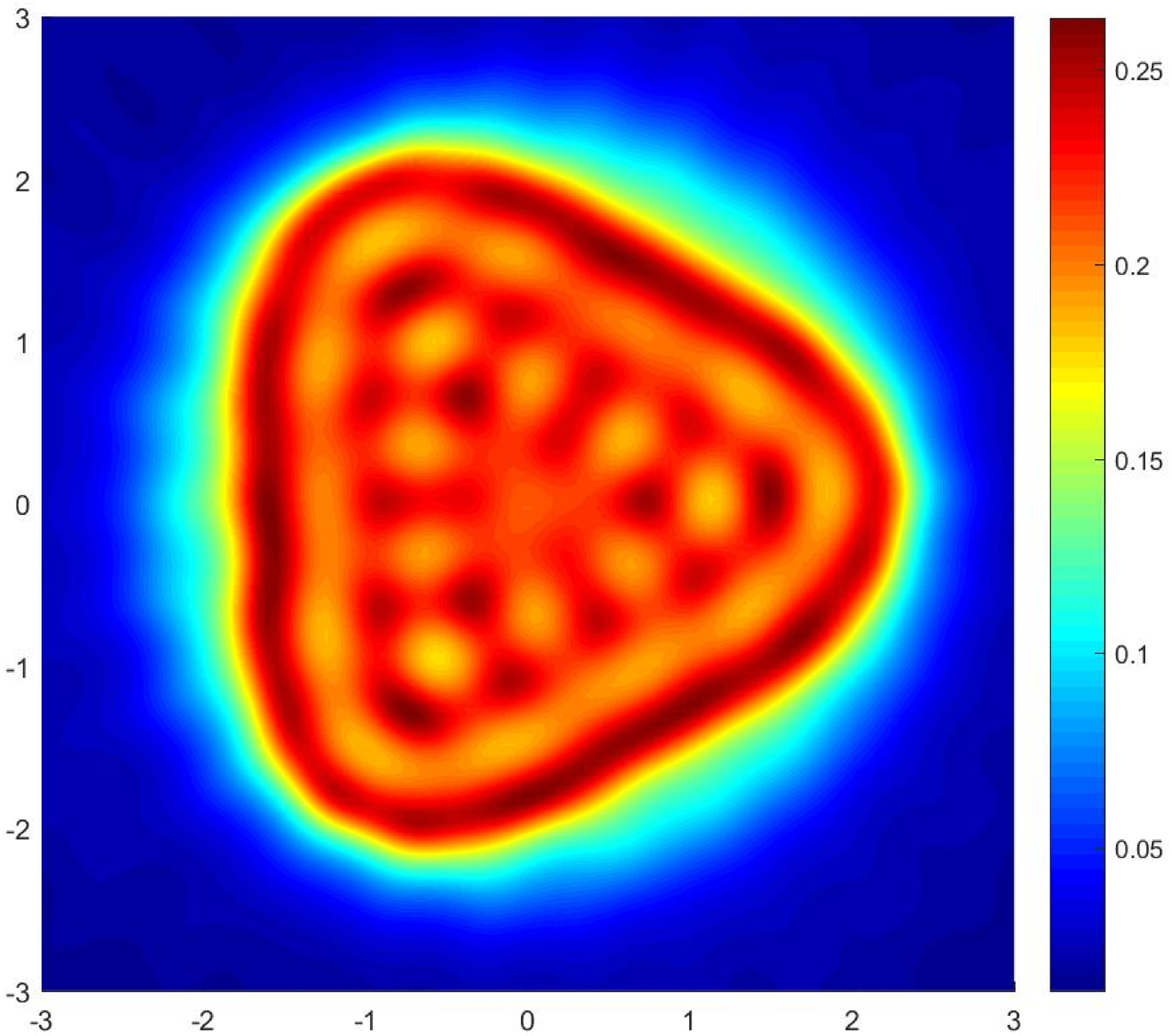}
(d) k=5, 10\% noise
\end{minipage}
\caption{Reconstruction of rounded triangle-shaped boundary $\pa D$.
The refractive index is given by $n(x)=2+2i$ in $D_1$ and $n(x)=0.5+2i$ in $D_2$.
}\label{fig1}
\end{figure}

\textbf{Example 2.} In this example, we consider the case when
$\pa D$ is a rounded square-shaped boundary and the  refractive index $n$ in $D$ is given by
\ben
n(x)=\left\{
  \begin{array}{ll}
    2+2i & \text{for}\;x\in D_1,\\
    0.5+2i & \text{for}\; x\in D_2,
  \end{array}
\right.
\enn
where $D_1=\{(x_1,x_2)\in D: x_1<0\}$ and $D_2=\{(x_1,x_2)\in D: x_1>0\}$.
See Figure \ref{fig2}(a) for
the physical configuration.
The reconstruction results of the boundary $\pa D$ are presented in Figure \ref{fig2}
by using the far-field data without noise, with 5\% noise
and with 10\% noise, respectively.

\begin{figure}[htbp]
\begin{minipage}[t]{0.4\linewidth}
\centering
\includegraphics[width=3in]{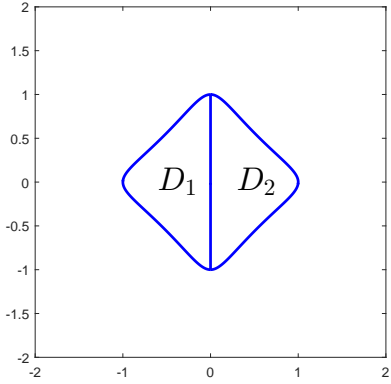}
(a) Physical configuration
\end{minipage}\qquad\qquad
\begin{minipage}[t]{0.4\linewidth}
\centering
\includegraphics[width=3in]{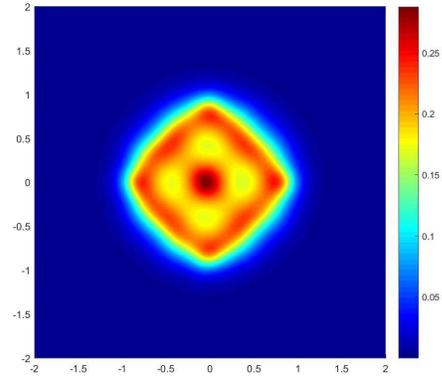}
(b) k=5, no noise
\end{minipage}
%
\begin{minipage}[t]{0.4\linewidth}
\centering
\includegraphics[width=3in]{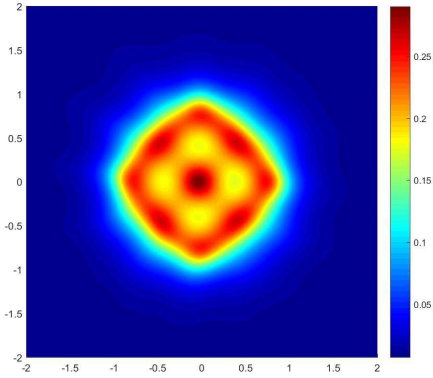}
(c) k=5, 5\% noise
\end{minipage}\qquad\qquad
\begin{minipage}[t]{0.4\linewidth}
\centering
\includegraphics[width=3in]{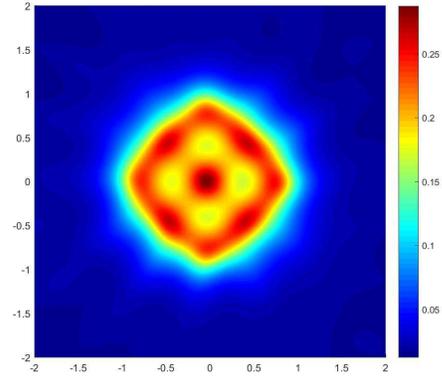}
(d) k=5, 10\% noise
\end{minipage}
\caption{Reconstruction of rounded square-shaped boundary $\pa D$.
The refractive index is given by $n(x)=2+2i$ in $D_1$ and $n(x)=0.5+2i$ in $D_2$.
}\label{fig2}
\end{figure}

\textbf{Example 3.} In this example, we consider the case when
$\pa D$ is a kite-shaped boundary and the refractive index $n$ in $D$ is given by
\ben
n(x)=\left\{
  \begin{array}{ll}
    0.5+2i & \text{for}\;x\in D_1,\\
    2+2i& \text{for}\; x\in D_2,
  \end{array}
\right.
\enn
where $D_1=\{(x_1,x_2)\in D: x_2>0\}$ and $D_2=\{(x_1,x_2)\in D: x_2<0\}$.
See Figure \ref{fig3}(a) for
the physical configuration.
The reconstruction results of the boundary $\pa D$ are presented in Figure \ref{fig3}
by using the far-field data without noise, with 5\% noise
and with 10\% noise, respectively.
\begin{figure}[htbp]
\begin{minipage}[t]{0.4\linewidth}
\centering
\includegraphics[width=3in]{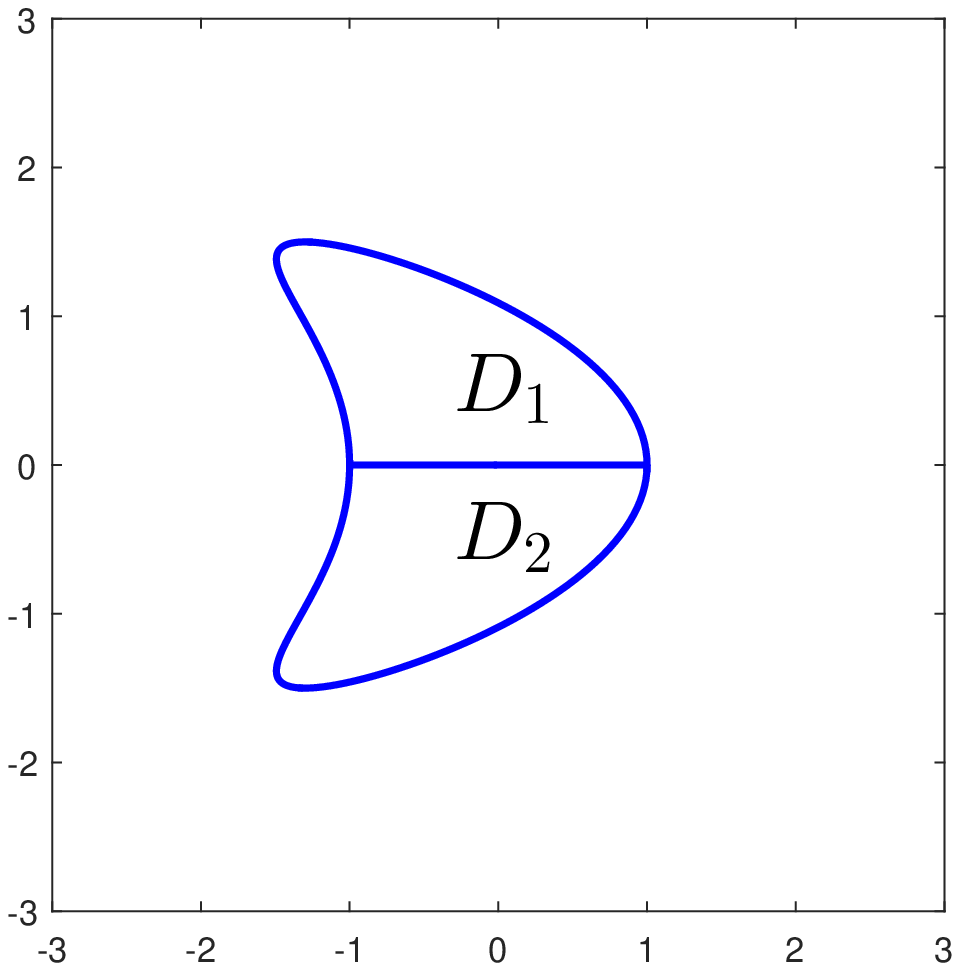}
(a) Physical configuration
\end{minipage}\qquad\qquad
\begin{minipage}[t]{0.4\linewidth}
\centering
\includegraphics[width=3in]{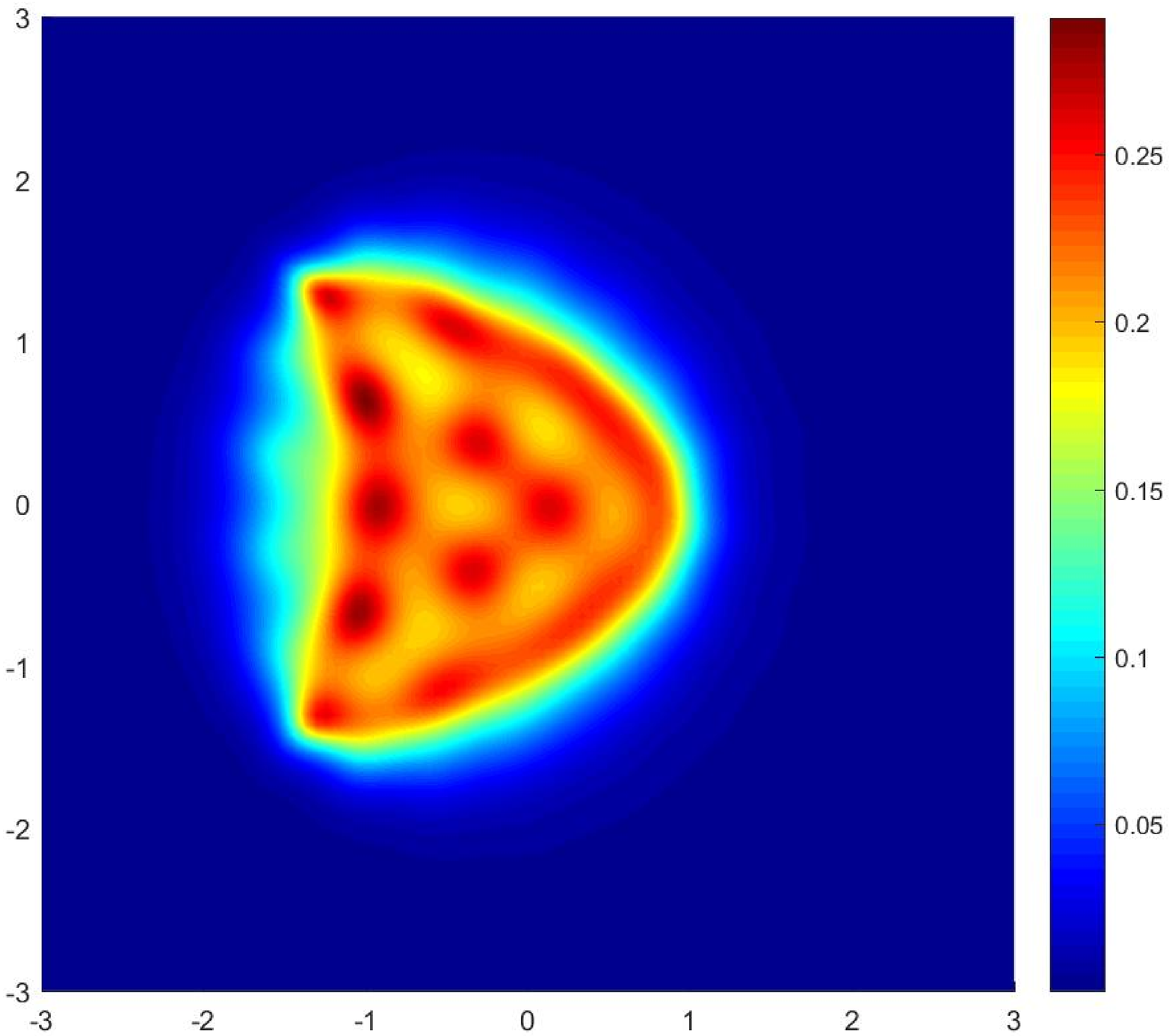}
(b) k=5, no noise
\end{minipage}
%
\begin{minipage}[t]{0.4\linewidth}
\centering
\includegraphics[width=3in]{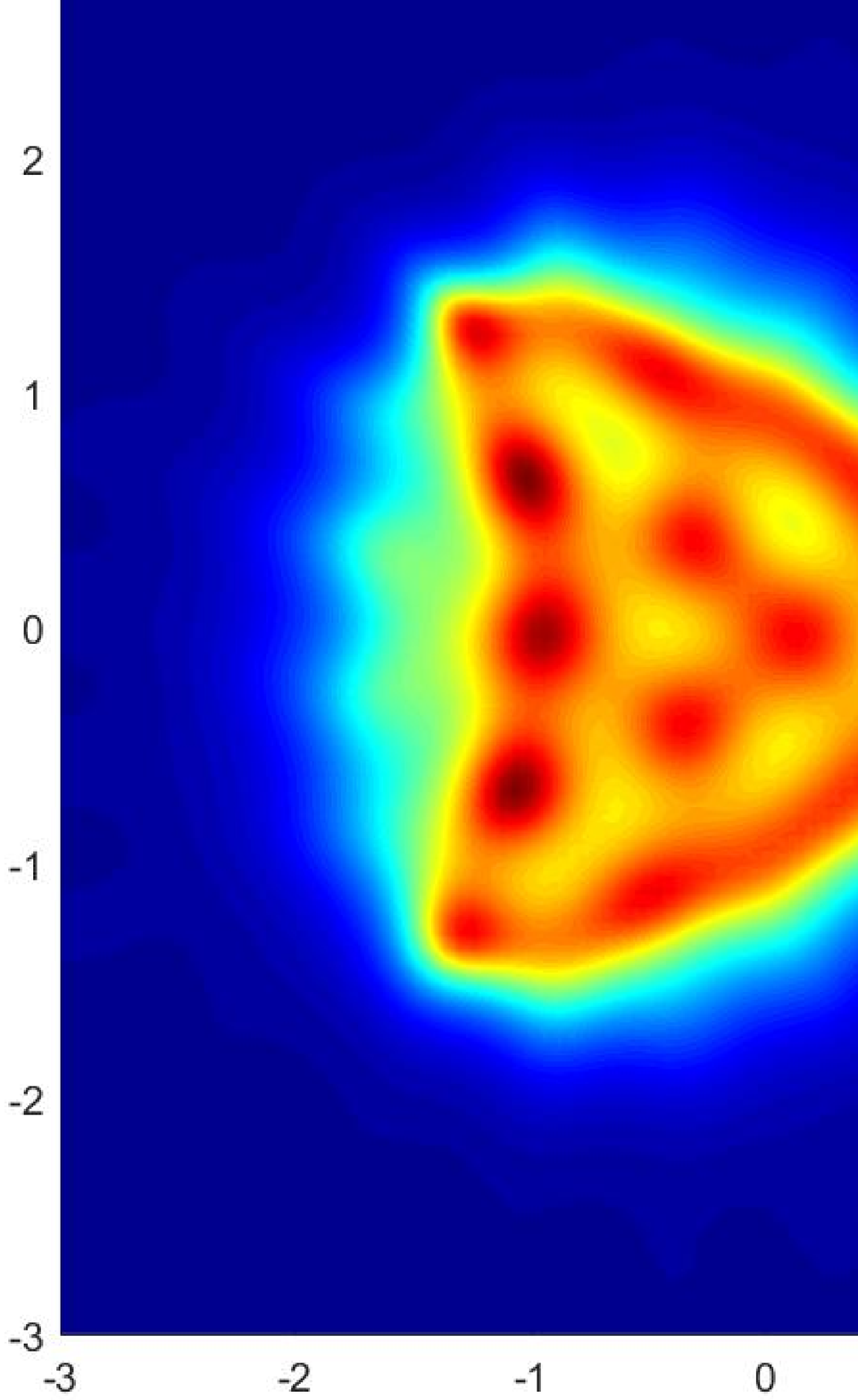}
(c) k=5, 5\% noise
\end{minipage}\qquad\qquad
\begin{minipage}[t]{0.4\linewidth}
\centering
\includegraphics[width=3in]{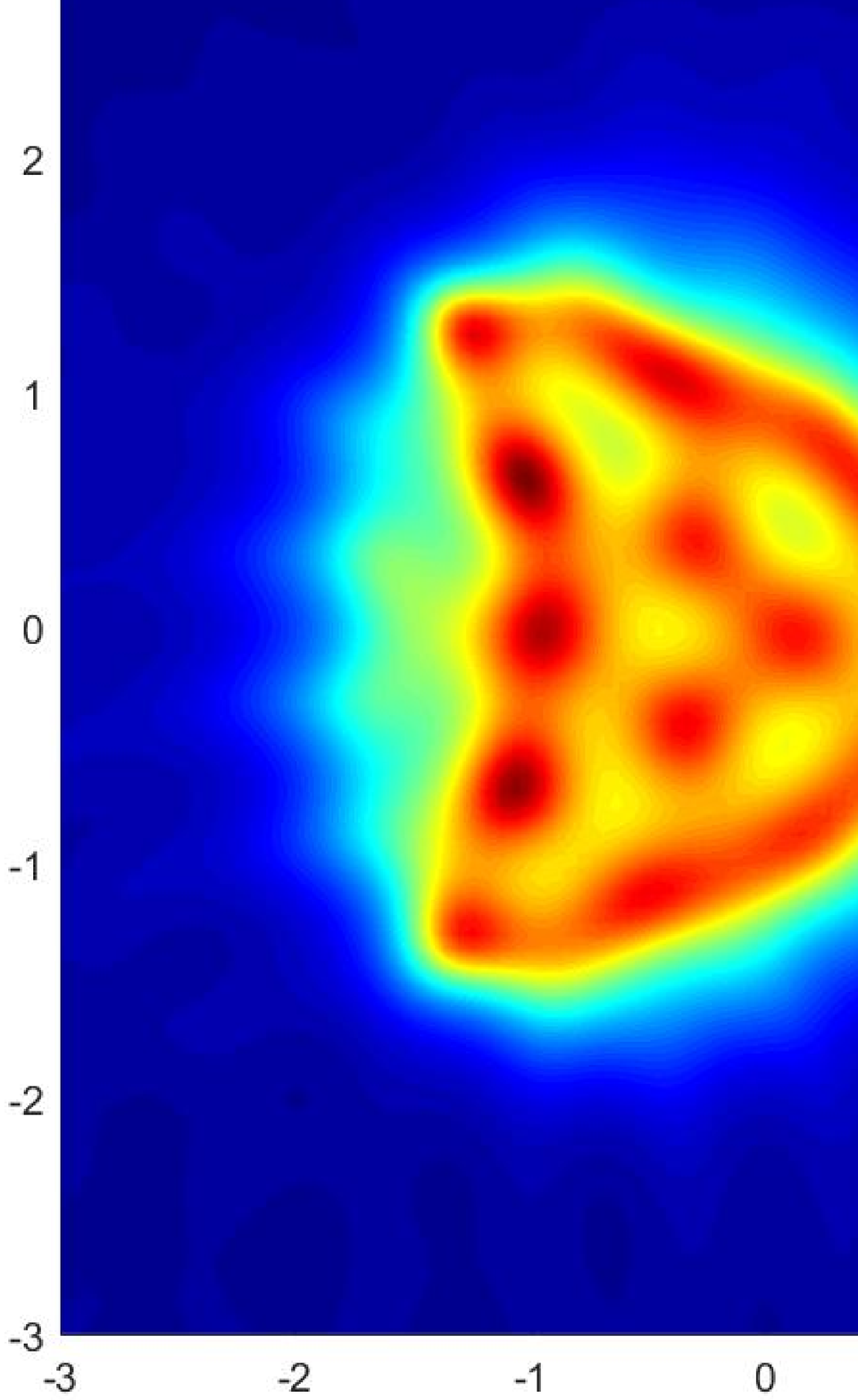}
(d) k=5, 10\% noise
\end{minipage}
\caption{Reconstruction of kite-shaped boundary $\pa D$.
The refractive index is given by $n(x)=0.5+2i$ in $D_1$ and $n(x)=2+2i$ in $D_2$.
}\label{fig3}
\end{figure}

{\color{hw}From the above three examples and the other cases carried out but not presented here it can be seen that
the shape and location of the obstacle $D$ is
numerically reconstructed from the spectral data of the far-field operator for the case of an inhomogeneous medium with complex refractive index. This indeed verifies the theoretical analysis of the approximate factorization method that presented in Section \ref{se2}. In the future, motivated by this work,
we hope to investigate the exact factorization of the far-field operator. Moreover, we plan to extend our result to the case of electromagnetic scattering problems, which is more challenging.}

\section*{Acknowledgements}
The work of F. Qu was supported by the NNSF of China Grant No.11871416, 11401513 and NSF of Shandong Province of China Grant No. ZR2017MA044. The work of H. Zhang was supported by the NNSF of China Grant No. 11501558 and 11871466.


\end{document}